\documentclass[10pt]{amsart}

\usepackage{a4wide}
\usepackage{amsthm,amsmath,amssymb,amscd}
\usepackage{latexsym}
\usepackage[all]{xy}
\usepackage{comment}
\usepackage{appendix}
\usepackage{pxfonts}
\usepackage{pdfsync}

\usepackage{graphicx}
\usepackage{wrapfig}
\usepackage{caption}

\usepackage{color}
\definecolor{Bordeaux}{rgb}{0.545, 0.137, 0.137}
\definecolor{BleuGris}{rgb}{0.212, 0.392, 0.545}
\definecolor{Chocolat}{rgb}{0.36, 0.2, 0.09}
\definecolor{BleuTresFonce}{rgb}{0.215, 0.215, 0.36}

\DeclareSymbolFont{rsfscript}{OMS}{rsfs}{m}{n}
\DeclareSymbolFontAlphabet{\mathrsfs}{rsfscript}

\DeclareFontFamily{OMS}{rsfs}{\skewchar\font'177}
\DeclareFontShape{OMS}{rsfs}{m}{n}{%
      <5> rsfs5
      <6> <7> rsfs7
      <8> <9> <10> rsfs10
      <10.95> <12> <14.4> <17.28> <20.74> <24.88> rsfs10
      }{}

\usepackage[colorlinks,final,backref=page,hyperindex]{hyperref}
\hypersetup{citecolor=BleuTresFonce, linkcolor=Chocolat, urlcolor=BleuTresFonce}

\theoremstyle{plain}
\newtheorem{theorem}{Theorem}[section]
\newtheorem*{inttheorem}{Theorem}
\newtheorem{lemma}{Lemma}[section]
\newtheorem{proposition}{Proposition}[section]
\newtheorem{corollary}{Corollary}[section]
\theoremstyle{definition}
\newtheorem{definition}{Definition}[section]
\theoremstyle{remark}
\newtheorem{remark}{Remark}[section]

\DeclareMathOperator{\Ord}{\sf Ord}
\DeclareMathOperator{\Vect}{\sf Vect}

\DeclareMathOperator{\Tw}{Tw}
\DeclareMathOperator{\Kos}{Kos}
\DeclareMathOperator{\Hom}{Hom}
\DeclareMathOperator{\End}{End}
\DeclareMathOperator{\Tor}{Tor}
\DeclareMathOperator{\Ext}{Ext}
\DeclareMathOperator{\id}{id}
\DeclareMathOperator{\lt}{lt}

\def\A{\mathrsfs{A}}

\def\BB{\mathrm{B}}
\def\T{\mathrsfs{T}}
\def\I{\mathrsfs{I}}

\def\P{\mathrsfs{P}}
\def\Q{\mathrsfs{Q}}
\def\NA{\mathrsfs{N\!\!A}}
\def\KK{\mathbb{K}}
\newcommand{\ac}{\scriptstyle \text{\rm !`}}

\title{Higher Koszul duality for associative algebras}

\author{Vladimir Dotsenko}
\address{School of Mathematics, Trinity College, Dublin 2, Ireland}
\email{vdots@maths.tcd.ie}

\author{Bruno Vallette}
\address{Laboratoire J.A.Dieudonn\'e, Universit\'e de Nice Sophia-Antipolis, Parc Valrose, 06108 Nice Cedex 02, France}
\email{brunov@unice.fr}

\begin{document}

\begin{abstract}
We present a unifying framework for the key concepts and results of higher Koszul duality theory for $N$-homogeneous algebras: the Koszul complex, the candidate for the space of syzygies, and the higher operations on the Yoneda algebra. 
We give a universal description of the Koszul dual algebra under a new  algebraic structure. For that we introduce a general notion: Gr\"obner bases for algebras over non-symmetric operads. 
\end{abstract}

\maketitle

\tableofcontents

\section*{Introduction}
The Koszul duality theory originates from the work of S. Priddy \cite{Priddy70}, inspired by ideas of Koszul \cite{Koszul50}. Priddy's theory describes ``small'' resolutions for associative algebras. To compute derived functors, like the $\Tor$ and the $\Ext$ functors, one needs projective resolutions of modules. There always exist such resolutions, which are built functorially from the bar construction, for instance. This functoriality has a price: such a resolution is pretty big. For a given algebra, one can find a (generally) smaller resolution by describing step by step an economic set of generators, called the syzygies. This inductive process is now called the Koszul--Tate resolution after \cite{Koszul50, Tate57}. The same type of process allows one to build a model for a given algebra~$A$, that is a differential graded free algebra whose homology is isomorphic to~$A$. Such a model, according to the general results of homotopical algebra~\cite{Quillen67}, can be used to compute the Quillen homology $H^Q_\bullet(A)$, for instance.

The original Koszul duality theory developed by Priddy starts from an algebra~$A$ given with a quadratic presentation $(V, R\subset V^{\otimes 2})$ and produces a quadratic Koszul dual coalgebra~$A^{\ac}$ generated by the same presentation. Under this algebraic structure, it produces, at once, a good candidate for the space of syzygies. All the three graded vector spaces $\Tor^A_\bullet(\KK,\KK)$, $\Ext^\bullet_A(\KK,\KK)$, $H^Q_\bullet(A)$ can be expressed naturally in terms of~$A^{\ac}$ by re-grading or passing to the graded dual $A^!:=(A^{\ac})^*$. The algebra $A^!$ admits a simple presentation as $T(V^*)/(R^\perp)$. The relationship between the differentials in the corresponding resolutions is more subtle.  

The higher case, when a given algebra~$A$ has a homogeneous presentation $(V, R\subset V^{\otimes N})$ of weight~$N>2$,  has been first studied by R. Berger in \cite{Berger01} (see also \cite{BDVW03,YuZhang2003}), where a candidate for the dual $A^!$ of the space of syzygies was described,
and where a condition for that candidate to work was given. 

\medskip

Each of the equivalent ways to look at the space of syzygies ($\Ext$, $\Tor$, Quillen homology) exhibits a rich structure of an 
associative algebra up to homotopy, or $\A_\infty$-algebra. This viewpoint has become more common in ring theory in the past decade, once it was understood that higher operations on the $\Ext$-algebras of associative algebras allow one to keep track of all the homotopy information of those algebras that would otherwise be lost if we restrict ourselves to the Yoneda product only. This circle of ideas is presented 
 in a sequence of papers of Lu, Palmieri, Wu, and Zhang \cite{LPWZ04,LPWZ09}.

In the case of $N$-Koszul algebras, the  $\A_\infty$-algebra structure of the Yoneda algebras was computed by He and Lu in~\cite{HeLu05}. They proved that most of the higher operations for such an algebra vanish, and that the non-vanishing ones admit very explicit formulas.
In a sense, their theorems provide an exhaustive description of the $\Ext$-algebras as $\A_\infty$-algebras. However, those theorems have a slightly mysterious feature which makes one hope for a different way to view the same result. Namely, in the quadratic case the $\Ext$-algebra of a Koszul algebra is an $\A_\infty$-algebra with all the higher operations vanishing, that is an associative algebra, presented by generators and relations.
In the $N$-Koszul case, the $\Ext$-algebra  carries an $\A_\infty$-algebra structure with all the higher operations except for the $N$th one vanishing, that is, an $\A_{2,N}$-algebra. It is also proved to be a reduced $\A_{2,N}$-algebra, meaning that it satisfies additional vanishing conditions. However to describe the higher operations explicitly, this algebra has to be realised as a subspace of a usual associative algebra.

\medskip

In this paper, we suggest to replace the condition of being reduced by a more functorial condition involving properties of operations only. More precisely, we consider a new type of algebras which we call $\NA_{2,N}$-algebras. An $\NA_{2,N}$-algebra is a graded vector space with an associative product $\mu_2$ of degree zero, and an $N$-ary operation $\mu_N$ of degree $2-N$  which satisfy the identity 
\begin{multline}\label{B-on-elements}
\mu_2(\mu_N(-,\ldots,-),-)+(-1)^{N-1}\mu_2(-,\mu_N(-,\ldots,-))+\\+
\sum_{i=1}^N(-1)^{i-1+N}\mu_N(-,\ldots,-,\mu_2(-,-),-,\ldots,-)=0,
\end{multline}
and the identities
\begin{equation}
\mu_N(-,\ldots,-,\mu_N(-,\ldots,-),-,\ldots,-)=0.  
\end{equation}
It turns out that the Berger's candidate $A^!$ for the linear dual of space of syzygies admits the following  presentation via generators and relations
as a $\NA_{2,N}$-algebra. 
\begin{inttheorem}[\ref{thm:presentationKDinfty}]
Let $A=T(V)/(R)$ be a finitely generated $N$-homogeneous algebra.
Then
 $$
A^!\cong\mathop{\mathrm{Free}}\nolimits_{\NA_{2,N}}(V^*)/(\mu_2(V^*,V^*),\mu_N(R^\bot)) 
 $$
as  $\NA_{2,N}$-algebras, where $\mu_N(R^\bot)$ is viewed as a subspace of $\mu_N(V^*,V^*,\ldots,V^*)$.
\end{inttheorem}
To prove this theorem, and in general to get a better understanding of $\NA_{2,N}$-algebras, we introduce a general notion of Gr\"obner bases for algebras over non-symmetric operads. 

We also exhibit a precise equivalence between describing a free $A$-module resolution of the ground field $\KK$ and describing a model for~$A$. Overall, this demonstrates  a clear analogy between the $N=2$ Koszul duality theory of Priddy and the higher Koszul property for $N$-homogeneous algebras with $N>2$. 

\medskip

We were motivated to develop the higher Koszul duality theory for $N$-homogeneous algebras in a way similar to the Koszul duality for quadratic algebras in order to generalise it even further. Some related work was done in \cite{GM11} for relations of weights $2$ and $N$, and in \cite{RS} for relations of weights~$N_1$ and $N_2$, both greater than~$2$.
However, the results of \cite{CG11} demonstrate that even ``small'' resolutions may yield  ``big'' $\A_\infty$-algebra structures. And so it might make sense to refine the corresponding class of Koszul algebras. 
We shall address this question in a future work. 

\medskip
 
\subsection*{Layout. }
The paper is organised as follows. In Section~\ref{sec:recoll}, we recall some basic notions and results we use throughout the paper. In Section~\ref{sec:Grobner}, we introduce Gr\"obner bases for non-symmetric operads and for algebras over a non-symmetric operad, and we formulate the key results about them. 
In Section~\ref{sec:operadB}, we introduce $\NA_{2,N}$-algebras and compute its reduced Gr\"obner basis.
 In Section~\ref{sec:present}, we prove the theorem stated above, and use it to formulate an equivalent criterion for an $N$-homogeneous algebra to be $N$-Koszul. In Section~\ref{sec:higherKoszul}, we relate the construction of the free resolution of the trivial $A$-module and the construction of a dg algebra resolution of~$A$, using the language of twisting morphisms.

\subsection*{Conventions. }
We work over a ground field $\KK$. The word ``algebra'', unless otherwise specified, refers to an associative algebra. As above, $N$ is a positive integer  greater than~$2$. Most of vector spaces that we use are bi-graded with finite-dimensional bi-graded components. 
The \emph{weight} grading is coming from the convention that we consider algebras with generators of weight~$1$ subject to homogeneous relations, and the \emph{(homological) degree} is coming from the fact that our constructions produce chain complexes. 

\subsection*{Acknowledgements. } The first author acknowledges partial support of Grant GeoAlgPhys 2011--2013 awarded by the University of Luxembourg. We would like to express our appreciation to the Max-Planck Institute for Mathematics in Bonn, where the work on this paper began, and to the University Lyon~1, where the paper was completed,  for the joint invitation and the excellent working conditions. We are grateful to Eric Hoffbeck for many useful comments on the first draft of this paper. Last but not the least, the first author wishes to thank the organisers of the workshop ``New developments in noncommutative algebra'' for a great opportunity to discuss some maths in the most beautiful scenery of the Isle of Skye; he benefited from a few conversations with participants of that workshop, and is especially thankful to Andrea Solotar for explaining the key results of~\cite{RS}. 
 
\section{Recollections}\label{sec:recoll}

We shall recall the basic definitions and results used throughout the paper, however trying to keep this section reasonably concise. For details on associative algebras and operads and their homotopy theory we refer the reader to the book \cite{LodayVallette10}. 

\subsection{Koszul $N$-homogeneous algebras}\label{subsec:N-Koszul}

\begin{definition}
An associative algebra $A$ presented by generators and relations, $A=T(V)/(R)$ is said to be \emph{$N$-homogeneous} if $R\subset V^{\otimes N}$. In this case, the algebra $A$ is weight graded: $A\cong \oplus_{m\ge 0} A_m$. When $V$ is finite dimensional, the \emph{$N$-homogeneous dual algebra} $A^\vee$ of an $N$-homogenenous algebra is defined by the formula 
 $$
A^\vee:=T(V^*)/(R^\bot),
 $$ 
where $R^\bot\subset (V^*)^{\otimes N}$ is the annihilator of $R$ under the natural pairing of vector spaces $(V^*)^{\otimes N}\otimes V^{\otimes N}\to\KK$. 
\end{definition}

In the case $N=2$, it turns out that various homological questions of $A$ admit natural answers in terms of the algebra $A^\vee$ which is, in that case, usually denoted by $A^!$. However, in the case $N>2$, only some of the homogeneous components of $A^\vee$ are relevant. We define the graded vector space $A^!$ by the formula
 $$
A^!_m:=
\begin{cases}
A^\vee_m, \quad m\equiv 0,1\pmod{N},\\
0,\phantom{aaaa} \text{otherwise}, 
\end{cases}
 $$
where in addition to being of weight $m$, $A^\vee_m$ is assigned the homological degree $\frac{2m}{N}$ if $m\equiv 0\pmod{N}$ and the homological degree $\frac{2(m-1)}{N}+1$ is $m\equiv1\pmod{N}$. The associative product on $A^\vee$ induces an associative product on $A^!$ (whenever the product in $A^\vee$ ends up in a component absent in $A^!$, the corresponding product is defined to be equal to zero). The algebra $A^!$ is called the \emph{Koszul dual algebra} of~$A$. Note that our notational convention differs somewhat from the one adopted in earlier literature \cite{Berger01,BDVW03}; we believe that it is strongly beneficial to use, for any homogeneous algebra $A$, the notation $A^{\ac}$ for the space of syzygies of~$A$, hence the change.

We also put $A^{\ac}:=(A^!)^*$ (we use the convention that the dual of a graded space with finite-dimensional component is the direct sum of component-wise duals). Since $A^!$ is weight-wise finite dimensional, $A^{\ac}$ is a coalgebra. Note that for $N=2$, we have $A^\vee=A^!$.

Recall from Berger \cite{Berger01} that the \emph{Koszul complex} of an $N$-homogeneous algebra $A$ is $A^{\ac}\otimes A$ with the boundary maps 
\begin{gather*}
A^{\ac}_{kN}\otimes A\to A^{\ac}_{(k-1)N+1}\otimes A,\\ 
A^{\ac}_{kN+1}\otimes A\to A^{\ac}_{kN}\otimes A 
\end{gather*}
being, respectively, the composites
\begin{gather*}
A^{\ac}_{kN}\otimes A\to A^{\ac}_{(k-1)N+1}\otimes (A^{\ac}_1)^{\otimes(N-1)}\otimes A\cong A^{\ac}_{(k-1)N+1}\otimes (A_1)^{\otimes(N-1)}\otimes A\to A,\\ 
A^{\ac}_{kN+1}\otimes A \to A^{\ac}_{kN}\otimes A^{\ac}_1\otimes A\cong A^{\ac}_{(k-1)N+1}\otimes A_1\otimes A\to A,
\end{gather*}
the maps coming from the coproduct of $A^{\ac}$, the identification $A^{\ac}_1=(V^*)^*\cong V=A_1$, and the product of $A$. Note that we assume elements of~$A$ to be of homological degree~$0$, so the elements of $A^{\ac}_{kN}\otimes A$ (respectively, $A^{\ac}_{kN+1}\otimes A$) are, according to the convention we adopted above, assigned the homological degree $2k$ (respectively, $2k+1$).

\begin{definition}
An $N$-homogeneous algebra $A$ is said to be \emph{$N$-Koszul} if its Koszul complex is acyclic in positive homological degrees.
\end{definition}

\subsection{Homotopy associative algebras}

Let $A$ be a graded vector space, and $f\colon A^{\otimes k} \to A$, $g\colon A^{\otimes l} \to A$ be two linear maps of degrees~$p$ and $q$ respectively. We define their pre-Lie product $f \star g \colon A^{k+l-1} \to A$ by the formula 
 $$
f \star g := \sum_{i=1}^k (-1)^{q(k-1)+(l-1)(i-1)} f\circ_i g \ ,
 $$
where $f\circ_i g$ stands for $f(\id^{\otimes(i-1)}\otimes g\otimes \id^{\otimes (k-i)})$.

\begin{definition}
Let $N\ge 3$ be an integer. An \emph{$\A_{2,N}$-algebra} is a triple $(A; \mu_2, \mu_N)$ where $A$ is a graded vector space and where $\mu_2 : A^{\otimes 2}\to A$ and $\mu_N : A^{\otimes N} \to A$ are linear maps of degree $0$ and $2-N$ respectively. They are required to satisfy the relations 
 $$
\mu_2 \star \mu_2 = \mu_2 \star\mu_N + \mu_N \star\mu_2=\mu_N \star \mu_N=0. 
 $$
\end{definition}
An $\A_{2,N}$-algebra is a particular case of an $\A_\infty$-algebra (see, e.~g., \cite{Keller01} and~\cite[Chapter~$9$]{LodayVallette10}), whose differential vanishes and whose structure maps $\mu_n$ vanish as well for $n\ne 2, N$ .

Dualising the above definitions, we can define the corresponding notion of coalgebras. If $C$ is a graded vector space, and $f\colon C\to C^{\otimes k}$, $g\colon C\to C^{\otimes l}$ be two linear maps of degrees~$p$ and $q$ respectively, we define their pre-Lie product $f \star g \colon C\to C^{k+l-1}$ by the formula 
 $$
f \star g := \sum_{i=1}^k (-1)^{q(k-1)+(l-1)(i-1)} g\ {}_i\!\circ f \ , 
 $$
where $g\ {}_i\!\circ f$ stands for $(\id^{\otimes( i-1)}\otimes g \otimes \id^{\otimes (k-i)})\circ f$. An \emph{$\A_{2,N}$-coalgebra} is a triple $(C; \delta_2, \delta_N)$ where $C$ is a graded vector space and where $\delta_2\colon C \to C^{\otimes 2}$ and $\delta_N\colon C \to  C^{\otimes N}$ are linear maps of degree $0$ and $N-2$ respectively. They are required to satisfy the relations 
 $$
\delta_2 \star \delta_2 = \delta_2 \star\delta_N + \delta_N \star\delta_2=\delta_N \star \delta_N=0. 
 $$

\subsection{Koszul dual and  $\Ext$-algebra}
This section is a brief summary of the results obtained in~\cite{HeLu05}.

Let $A$ be an augmented algebra. Recall that the $\Ext$-algebra $\Ext^\bullet_A(\KK, \KK)$ of $A$ is defined as the derived functor $R^\bullet\Hom_A(\KK,\KK)$; it can be computed in many ways, for instance, via the bar construction. The \emph{bar construction} $\BB A$ is a quasi-free coalgebra $T^c(s\bar A)$ on the suspension of the augmentation ideal of $A$. Its differential is the unique coderivation extending the product of $A$. Hence, the linear dual $(\BB A)^*$ of the bar construction is a differential graded algebra, whose underlying cohomology groups are equal to the $\Ext$-functor: 
$$H^\bullet ((\BB A)^*)\cong \Ext^\bullet_A(\KK,\KK)\ .$$
As the homology of a dg algebra, the $\Ext$-functor acquires a canonical associative product; this algebra is called the \emph{Yoneda algebra}. Moreover, using the Homotopy Transfer Theorem \cite{LPWZ09, LodayVallette10}, one can endow the $\Ext$-functor with an $\A_\infty$-algebra structure, which extends the Yoneda product. We call it the \emph{Yoneda $\A_\infty$-algebra}. In the $\A_\infty$-setting, we shall, for simplicity, mostly work with non-unital algebras, and therefore consider the augmentation
ideal of the usual Yoneda algebra.

\begin{definition}
An $\A_{2,N}$-algebra $E$ is said to be \emph{reduced} if $\mu_2(a_1,a_2)=0$ when both $a_1$ and $a_2$ are of odd degree, and $\mu_N(a_1,\ldots,a_N)=0$ when at least one of $a_i$ is \emph{not} of odd degree. 
\end{definition}

One of the reasons for the notion of a reduced $\A_{2,N}$-algebra to be important is explained by the following theorem.

\begin{theorem}[{\cite[Thm.~6.2]{HeLu05}}]\label{thm:reduced}
Let $A$ be an $N$-homogeneous algebra. It is $N$-Koszul if and only if its Yoneda $\A_\infty$-algebra is a reduced $\A_{2,N}$-algebra generated in degree~$1$.
\end{theorem}

The higher operations of the Yoneda algebra of an $N$-Koszul algebra can be described as follows. 

\begin{theorem}[{\cite[Thm.~6.5]{HeLu05}}]\label{computeproducts}
The Yoneda $\A_\infty$-algebra $\Ext^\bullet_A(\KK,\KK)$ of an $N$-Koszul algebra $A$ satisfies 
 $$
\Ext^k_A(\KK,\KK)=
\begin{cases}
A^!_{N\frac{k}{2}}, \phantom{aaai}\text{ if }k\text{ is even},\\
A^!_{N\frac{k-1}{2}+1}, \text{ if }k\text{ is odd}.
\end{cases} 
 $$
The operations $\mu_2$ and $\mu_N$ on that algebra can be computed using the product on $A^\vee$ as follows:
\begin{gather*}
\mu_2 \colon 
\begin{cases}
{A^!}_{iN} \otimes  {A^!}_{jN} \to  {A^!}_{(i+j)N},    \\
{A^!}_{iN+1} \otimes  {A^!}_{jN} \to  {A^!}_{(i+j)N+1}, \\
{A^!}_{iN} \otimes  {A^!}_{jN+1} \to  {A^!}_{(i+j)N+1},
\end{cases}\\
\mu_N \colon {A^!}_{k_1 N+1} \otimes \cdots \otimes {A^!}_{k_N N+1} \to {A^!}_{(k_1+\cdots+k_N+1)N},   
\end{gather*}
and are zero for all other choices of arguments.
\end{theorem}

\subsection{Non-symmetric operads}

We denote by~$\Ord$ the category whose objects are finite ordered sets (with order-preserving bijections as morphisms), and by $\Vect$ the tensor category of graded vector spaces (by definition, morphisms in that category are degree zero linear operators, and the symmetry isomorphism is given by the formula $\sigma(v\otimes w)=(-1)^{|v||w|}w\otimes v$). A \emph{(non-symmetric) collection} is a contravariant functor from the category~$\Ord$ to the category~$\Vect$. We shall refer to images of individual sets as \emph{components} of our collection. We denote the set $[k]:=\{ 1, 2, \ldots, k \}$.

\begin{definition}
Let $\P$ and $\Q$ be two  collections. The \emph{composition product} of $\P$ and $\Q$ is the  collection $\P\circ\Q$ defined by the formula
\begin{equation}\label{composition}
(\P\circ\Q)(I):=\bigoplus_{k}\P([k])\otimes\left(\bigoplus_{f\colon I\twoheadrightarrow[k]}\Q(f^{-1}(1))\otimes\cdots\otimes\Q(f^{-1}(k))\right), 
\end{equation}
where the sum is taken over all non-decreasing surjections~$f$. 
\end{definition}

The composition product equips the category of  collections with a structure of a monoidal category. The unit object is the collection $\I$ with 
 $$
\I(M)=
\begin{cases}
\KK.M, \quad |M|=1,\\
0, \phantom{aaii}\quad |M|\ne 1. 
\end{cases} 
 $$
 
\begin{definition}
A \emph{non-symmetric operad} is a monoid in the monoidal category of  collections equipped with the composition product.
\end{definition}

By functoriality, for every collection $\P$ and every finite ordered set $M$, we have the corresponding isomorphism $\P(M)\cong\P(\{1,2,\ldots,|M|\})$, so to define a collection it is sufficient to define a graded vector space 
 $$
\P(n):=\P(\{1,2,\ldots,n\}), \quad \text{for}\  n\ge0 \ .
 $$
The component $\P(n)$ of this graded vector space may be viewed as the space of $n$-ary operations of some sort.

The associativity and the unit condition for a monoid ensure that to know the non-symmetric operad structure on a collection it is sufficient to know all the \emph{partial compositions}: those correspond to the surjections $f$ in Formula~\eqref{composition} (which we now apply in the case $\Q=\P$) for which all the preimages except for $f^{-1}(i)$ for some fixed~$i$ consist of one element, and such that $\P(f^{-1}(j)$ is equal to the unit, for $j\neq i$. Such a partial composition is denoted by $\alpha\circ_i\beta$. These partial compositions satisfy the appropriate associativity relations; namely, if $\alpha\in\P(n)$, $\beta\in\P(m)$, $\gamma\in\P(r)$, we have
\begin{equation}\label{graded-assoc}
(\alpha\circ_i\beta)\circ_j\gamma=
\begin{cases}
(-1)^{|\beta||\gamma|} (\alpha\circ_j\gamma)\circ_{i+r-1}\beta, \qquad  1\le j\le i-1,\\
\alpha\circ_i(\beta\circ_{i+j-1}\gamma), \phantom{aaaaaaaaaaaa} i\le j\le i+m-1,\\
(-1)^{|\beta||\gamma|} (\alpha\circ_{j-m+1}\gamma)\circ_i\beta, \phantom{aaaa}  i+m\le j\le n+m-1.
\end{cases}
\end{equation}
These associativity relations include signs coming from the symmetry isomorphisms; these signs will be of crucial importance in our subsequent computations.

Similarly to the case of associative algebras, a non-symmetric operad can be presented via generators and relations, that is, as a quotient of the free non-symmetric operad $\T(W)$ for some collection~$W$. Let us describe an explicit construction of~$\T(W)$. Suppose that $X$ is a collection of ordered sets, that is a functor from $\Ord$ to $\Ord$, that provides bases for the components of $W$. 

\begin{definition}
A \emph{(planar) tree monomial} is a planar rooted tree with labelled vertices: ``planar'' means that for each vertex $v$ of a tree, the set of its inputs~$I_v$ is ordered, and ``labelled'' means that  each vertex $v$ carries a label from the set $X(I_v)$. We use the notation $t(T)$ for the underlying tree of a tree monomial~$T$. 
\end{definition}

For example, if $X(n)=\varnothing$ for $n\ne 2$, then the component $\T(W)(3)$ has basis elements 
\begin{center}
\includegraphics[scale=0.9]{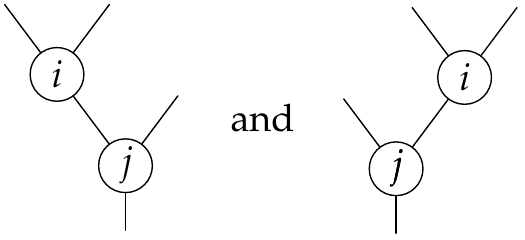}
\end{center}
indexed by pairs $i,j\in X(2)$. 

\begin{proposition}
The free non-symmetric operad $\T(W)$ admits a basis made up of (planar) tree monomials. 
\end{proposition}

An important property of the aforementioned basis is that each composition of tree monomials is again a tree monomial. To be precise, a tree monomial should be viewed as an equivalence class: to determine the actual composition product in the operad, one has to introduce levels of trees, i.~e. decide what should be composed first (ignoring that would only compute the result up to a sign). 

A basic example of a non-symmetric operad is the \emph{endomorphism operad} $\End_V$ of a vector space $V$, whose components $\End_V(n)$ is $\Hom(V^{\otimes n},V)$ with the partial composition of $f\in\End_V(k)$ and $g\in\End_V(l)$ given by the composition of functions: $f\circ_i g:=f(\id^{\otimes(i-1)}\otimes g\otimes \id^{\otimes (k-i)})$. By definition, a structure of an \emph{algebra} over a non-symmetric operad $\P$ on a vector space $V$ is a morphism of operads $\P\to\End_V$. The \emph{free $\P$-algebra} ${\P}(W)$ generated by a vector space~$W$ is given by the formula
 $$
{\P}(W):=\bigoplus_{n\ge 0}\P(n)\otimes W^{\otimes n}.
 $$

\section{Gr\"obner bases}\label{sec:Grobner}

In this section, we shall  define the notions of Gr\"obner bases for non-symmetric operads and their algebras. Some related work for ``free nonassociative algebras'' has been done in~\cite{Ger06}; a much more general approach, on which our exposition is modelled, is discussed in~\cite{DK10} within the framework of shuffle operads. Let us note, however, that there is a very important feature of non-symmetric operads that symmetric and shuffle operads do not exhibit: the machinery of Gr\"obner bases for non-symmetric operads includes $0$-ary operations. 

\subsection{Gr\"obner bases for non-symmetric operads}
The definitions and results below are completely parallel to those for associative algebras. Basically, Gr\"obner bases allow one to find for each algebra~$A$ a ``monomial replacement'', that is, an algebra $\mathring{A}$ with the same generators and with monomial relations for which the natural monomial basis (of monomials not divisible by any of relations) is also a basis for the vector space $A$.

\begin{definition}
A total ordering of the set of tree monomials is said to be \emph{admissible} if the operadic composition of several tree monomials is an increasing function in each argument.
\end{definition}

Let us describe one of the orderings of tree monomials that we shall use throughout the paper. This ordering was discovered and used by E.~Hoffbeck in~\cite{Hof10}.

\begin{definition}
The \emph{path-lexicographic ordering} of tree monomials is defined as follows. We first assign, to each tree monomial $T$, a sequence of words in the alphabet 
 $$
\mathbb{X}=\bigsqcup_{k\ge0} X(\{1,2,\ldots,k\}), 
 $$
one word for each leaf of~$t(T)$ (here by leaves we mean both the inputs and the vertices that have no inputs but however have labels corresponding to operations with no arguments). For each of those leaves, starting from the leftmost one and moving from left to right, we walk along the unique path from the root of~$t(T)$ to that leaf, and write down the labels of vertices encountered along the way. We form a word in the alphabet~$\mathbb{X}$ from these labels using the isomorphisms $X(I_v)\cong X(\{1,2,\ldots,|I_v|\})$. For each total ordering of~$\mathbb{X}$, we are now able to define the corresponding path-lexicographic ordering of tree monomials. To compare two monomials $S$ and $T$, we first compare the numbers of leaves; a tree with the larger number of leaves is, by definition, larger. If two trees have the same number of leaves, we compare the corresponding sequences word by word, comparing words degree-lexicographically. 
\end{definition}

For example, we have
\begin{center}
\includegraphics[scale=0.9]{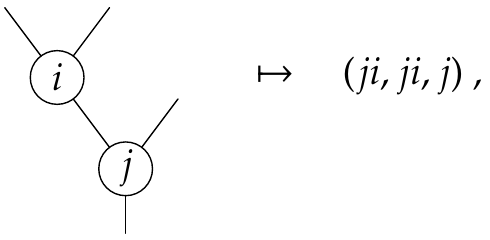} 
\end{center}
and therefore 
\begin{center}
\includegraphics[scale=0.9]{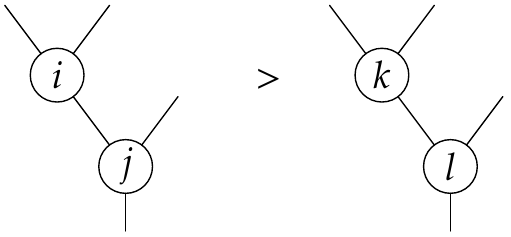}
\end{center}
whenever $(ji,ji,l)>(lk,lk,j)$, that is if $j>l$, or $j=l$ and $i>k$.

\begin{proposition}[\cite{Hof10}]
The path-lexicographic ordering is admissible.
\end{proposition}

In what follows, we fix some admissible ordering of tree monomials in $\T(W)$. All our definitions and statements are valid for any choice of such ordering. 

\begin{definition}
A tree monomial $S$ is said to be \emph{divisible} by a tree monomial~$T$ if $t(S)$ contains $t(T)$ as a subtree with matching labels. For a fixed admissible ordering, we define the \emph{leading term} $\lt(g)$ of an element $g\in\T(W)$ to be the largest tree monomial~$T$ that appears in $g$ with a non-zero coefficient. This nonzero coefficient (the leading coefficient of $g$) is denoted by $c_g$. 
\end{definition}

Assume that a tree monomial $S$ is divisible by a tree monomial~$T$. For every element $g\in\T(W)$ of the same arity as~$T$, we can consider the operation that inserts $g$ in the place of $T$ in~$S$; this produces an element of the free operad of the same arity as~$S$. We denote that operation by~$m_{S,T}(g)$. Note that by the construction $m_{S,T}(T)=S$, and since we work with admissible orderings, it is clear that if $T'<T$, then $m_{S,T}(T')<S$.

\begin{definition}
Assume that $f$ and $g$ are two elements of $\T(W)$ for which $\lt(f)$ is divisible by $\lt(g)$. The element
 $$
r_g(f):=f-\frac{c_f}{c_g}m_{\lt(f),\lt(g)}(g),
 $$
is called the \emph{reduction of $f$ modulo~$g$}. We shall often use the notation $f\rightarrow r_g(f)$ for reductions.
\end{definition}

\begin{definition}
Two tree monomials $S$ and $T$ are said to have a \emph{small common multiple}, if there exists a tree monomial $U$ which is divisible by both $S$ and $T$ for which each of its vertices is contained in either $S$ or $T$, and such that the number of vertices of $t(U)$ is less than the total number of vertices of $t(S)$ and $t(T)$. Assume that the leading terms of two elements $f,g\in\T(W)$ have a small common multiple $U$. 
The element
 $$
s_U(f,g):=m_{U,\lt(f)}(f)-\frac{c_f}{c_g} m_{U,\lt(g)}(g),
 $$
is called the \emph{$S$-polynomial of $f$ and $g$} (corresponding to the small common multiple $U$). 
\end{definition}

\begin{definition}
A \emph{Gr\"obner basis} of an ideal $I$ in the free non-symmetric operad is a system $G$ of generators of~$I$ for which the leading tree monomial of every element of~$I$ is divisible by one of the leading terms of elements of~$G$. A Gr\"obner basis $G$ is said to be \emph{reduced} if none of the leading terms of elements of $G$ are divisible by leading terms of other elements of~$G$.  
\end{definition}

One can prove that every ideal of the free non-symmetric operad admits a unique reduced Gr\"obner basis. 

\begin{definition}
A tree monomial is said to be \emph{normal} with respect to a system of elements $G$ if it is not divisible by any of the leading terms of elements of~$G$.
\end{definition}

It is easy to see that iterated reductions modulo elements of some set $G$ allow one to replace every element of a free operad by a linear combinations of tree monomials that are normal with respect to~$G$. In the case of a Gr\"obner basis, this representation produces what is natural to think of as a canonical form of an element, as the following proposition shows.

\begin{proposition}
Suppose that $G$ is a Gr\"obner basis of an ideal~$I\subset\T(W)$. Then normal tree monomials with respect to~$G$ form a basis of the quotient $\T(W)/I$. 
\end{proposition}

A useful criterion for a set of generators~$G$ to form a Gr\"obner basis is given by the following version of Bergman's Diamond Lemma~\cite{Bergman78}. The proof is analogous to that for Diamond Lemma and is omitted.

\begin{theorem}\label{DiamondLemma}
Suppose that $G$ is a system of generators of~$I$. Then $G$ forms a Gr\"obner basis if and only if each such S-polynomial of two elements of $G$ can be reduced to zero modulo~$G$. 
\end{theorem}

This theorem allows one to come up with an algorithm, analogous to Buchberger's algorithm in the commutative case~\cite{Buchberger65}, that computes for each operad presented by generators and relations its reduced Gr\"obner basis. In what follows, a computation of a Gr\"obner basis of a certain non-symmetric operad will be done from the beginning to the end, which will hopefully be instructional enough to the reader. 

\subsection{Gr\"obner bases for algebras over non-symmetric operads}
Unlike the case of symmetric operads where first steps in studying an operad can be done without symmetries by passing to shuffle operads~\cite{DK10} but for algebras over a given operad there is no clear way to define a theory of Gr\"obner bases, for non-symmetric operads one can handle the corresponding algebras as well. 

Let $\P=\T(W)/I$ be a non-symmetric operad, and $A$ be an algebra over $\P$. By a slight abuse of notation, we think of~$A$ as of a collection concentrated in arity~$0$. The following construction is an explicit adaptation for the non-symmetric case of the construction of the \emph{enveloping operad} for the pair $(\P,A)$, see~\cite{BM09} and references therein. 

\begin{definition}
The \emph{extension of constants in~$\P$ by~$A$}, denoted $\P\ltimes A$, is the operad 
 $$
\T(W\oplus A)/(I\oplus I_A), 
 $$
where $I_A$ consists of all the relations 
 $$
\mu\circ(a_1,\ldots,a_n)=\mu(a_1,\ldots,a_n)
 $$ 
with $\mu\in W(n)$, and the elements $a_1,\ldots,a_n\in A$ viewed as operations of arity~$0$ in $\P\ltimes A$ on the left-hand side (so that the non-symmetric composition $\mu\circ(a_1,\ldots,a_n)\in\T(W\oplus A)\circ\T(W\oplus A)$ can be evaluated), and as elements of the algebra~$A$ on the right-hand side (so that $\mu$ acts on these elements via the homomorphism $\P\to\End_A$).
\end{definition}

\begin{proposition}
We have $(\P\ltimes A)(0)\cong\P(0)\oplus A$.  
\end{proposition}

Consequently, it becomes clear how to develop Gr\"obner bases for algebras over any non-symmetric operad: one should pass to the extension of constants, compute the operadic Gr\"obner basis, and restrict it to arity~$0$. This results in the following theorem.

\begin{theorem}\label{basis-extension}
Let $\P=\T(W)/I$ be a non-symmetric operad, and let $A$ be an algebra over~$\P$. If $G$ is the reduced Gr\"obner basis of the operad~$\P\ltimes A$, then the normal tree monomials of arity zero not belonging to~$\P$ form a basis of~$A$.
\end{theorem}

In the case of the operad of associative algebras, this leads to the standard notion of a monomial replacement coming from the theory of Gr\"obner--Shirshov bases~\cite{Bokut76}, whereas in various other case this produces new definitions that are to be explored. To the best of our knowledge, even for the simplest case of algebras with two associative products satisfying no identities with each other, the corresponding Gr\"obner bases theory is not known.

\section{An operad acting on reduced $\A_{2,N}$-algebras}\label{sec:operadB}

The general approach of the operad theory is to express universal results in terms of relations that operations satisfy when they are evaluated on any elements. In this section, we suggest relations to provide a replacement for the property of an $\A_{2,N}$-algebra to be reduced. Namely, we introduce a quotient of the operad $\A_{2,N}$ that naturally acts on all reduced $\A_{2,N}$-algebras. We conclude with the description of its reduced Gr\"obner basis. 

\begin{definition}
The operad $\NA_{2,N}$ is generated by two operations, a binary operation $\mu_2$ and an $N$-ary operation $\mu_N$ that satisfy the identities
\begin{gather}
\mu_2\star\mu_2=0,\label{mu22}\\
\mu_2\star\mu_N+\mu_N\star\mu_2=0,\label{mu2N}\\
\mu_N\circ_i\mu_N=0,  \quad \text{for }\  i=1,\ldots,N \ .\label{muNN} 
\end{gather}
\end{definition}

\begin{proposition}
The operations of any reduced $\A_{2,N}$-algebra satisfy the relations of the operad $\NA_{2,N}$.
\end{proposition}

\begin{proof}
Indeed, $\mu_N(a_1,\ldots,a_N)=0$ when at least one of $a_i$ is not of odd degree. Also, in the case when all the elements $a_i$ are of odd degree, the element $\mu_N(a_1,\ldots,a_N)$ is of even degree. We immediately conclude that the operation $\mu_N\circ_i\mu_N$ is identically zero.
\end{proof}

\begin{remark}
For the condition of being a reduced $\A_{2,N}$-algebra in the sense of~\cite{LHL} (which is similar to the one used in this paper but includes congruences modulo~$3$) relevant for the concept of ``piecewise-Koszul algebras'', one can show in the exact same way that the corresponding algebras carry canonical $\NA_{2,N}$-algebra structures.
\end{remark}

\begin{corollary}\label{A2N-constr}
For every weight graded algebra $A$ concentrated in degree zero, we consider the vector space $E$ whose component of degree~$k$ is defined as
 $$
E^k=
\begin{cases}
A_{N\frac{k}{2}}, \phantom{aaai}\text{ if }k\text{ is even},\\
A_{N\frac{k-1}{2}+1}, \text{ if }k\text{ is odd}.
\end{cases} 
 $$
The operations $\mu_2$ and $\mu_N$ given by the product on $A$ by
\begin{gather*}
\mu_2 \colon 
\begin{cases}
A_{iN} \otimes A_{jN} \to A_{(i+j)N},    \\
A_{iN+1} \otimes A_{jN} \to A_{(i+j)N+1}, \\
A_{iN} \otimes  A_{jN+1} \to A_{(i+j)N+1},
\end{cases}\\
\mu_N \colon A_{k_1 N+1} \otimes \cdots \otimes A_{k_N N+1} \to A_{(k_1+\cdots+k_N+1)N}   
\end{gather*}
and defined to be zero for all other choices of arguments, endow $E$ with a $\NA_{2,N}$-algebra structure.
\end{corollary}

\begin{proof}
It is clear, since this algebra is easily seen to be a reduced $\A_{2,N}$-algebra. 
\end{proof}

\begin{corollary}
For every $N$-Koszul algebra, the operations on its Yoneda $\A_\infty$-algebra satisfy the relations of the operad~$\NA_{2,N}$. 
\end{corollary}

\begin{proof}
Indeed, by Theorem~\ref{computeproducts}, the Yoneda $\A_\infty$-algebra of an $N$-Koszul algebra~$A$ is the algebra $E$ obtained from the algebra $A^\vee$ by the construction we just described above.
\end{proof}

\begin{remark}
The Yoneda $\A_\infty$-algebra of the algebra $\KK\langle x,y\mid x^2=xy=y^N=0\rangle$ turns out to be a $\NA_{2,N}$-algebra which is not reduced as an $\A_{2,N}$-algebra (in any sense of the word ``reduced'' used in the literature), contrary to the feeling that the previous statements might create. This example suggests that the notion of a $\NA_{2,N}$-algebra is well suited for generalisations of the $N$-Koszul duality theory including relations of weight~$2$ and~$N$. This algebra is indeed $2$-$N$-Koszul in the sense of Green and Marcos~\cite{GM11}. We shall explore this observation further in a sequel paper.
\end{remark}

It turns out that Yoneda $\A_\infty$-algebras of $N$-Koszul algebras admit compact presentations as algebras over the operad $\NA_{2,N}$. In order to prove that, we shall study the operad $\NA_{2,N}$ in more detail, exhibiting its reduced Gr\"obner basis. In what follows, we use the path-lexicographic ordering on tree monomials, assuming that $\mu_2$ is greater than $\mu_N$ lexicographically.

\begin{theorem}\label{thm:GBB2N}
The reduced Gr\"obner basis of the operad $\NA_{2,N}$ is obtained from its defining relations \eqref{mu22}, \eqref{mu2N}, and \eqref{muNN} by adjoining, for each $k\ge 1$ and for each $i=1,\ldots,N-1$, the relation
\begin{gather}\label{tower}
\mu_N\circ_i(\mu_2^{(k)}\circ_{k+1}\mu_N)=(-1)^{N(i-1)}(\mu_N\circ_N(\mu_2\circ_2\mu_N))\circ_i\mu_2^{(k-1)},
\end{gather}
with the operations $\mu_2^{(k)}$ defined inductively by $\mu_2^{(0)}=\id$, $\mu_2^{(k+1)}=\mu_2(\id,\mu_2^{(k)})$.
\end{theorem}

\begin{proof}
By Theorem~\ref{DiamondLemma}, it is enough to prove that all S-polynomials coming from small common multiples of leading terms of the given relations can be reduced to zero using the same system of relations. We shall explain how the proof goes, and demonstrate it in the case of low arities, since in further arities all the computations one has to performed are modelled on those for low arities.

The S-polynomial that arises from the small common multiple the leading term $\mu_2\circ_1\mu_2$ of the relation $\mu_2\star\mu_2=0$ has with itself can be reduced to zero using just that relation; this can be interpreted in many different ways including the setup of the MacLane coherence axiom for monoidal categories; see \cite[Section~$8.1$]{LodayVallette10} for details. The S-polynomial that arises from the small common multiple of the leading terms $\mu_2\circ_1\mu_2$ and $\mu_2\circ_1\mu_N$ of the relations $\mu_2\star\mu_2=0$ and $\mu_2\star\mu_N+\mu_N\star\mu_2=0$ respectively can as well be reduced to zero using only those operations and nothing else; this is a result of a more tedious but straightforward computation. However, the leading term $\mu_2\circ_1\mu_N$ of the relation $\mu_2\star\mu_N+\mu_N\star\mu_2=0$ forms $n$ small common multiples with the relations $\mu_N\circ_i\mu_N$, and the corresponding S-polynomials cannot be reduced using only the defining relations of $\NA_{2,N}$; that is 
where new elements of the reduced Gr\"obner basis start showing up. We shall exhibit here the first steps of the respective computations, the further steps being completely analogous. In what follows, we use the notation $R_{i,k}$ ($k\ge1$, $i=1,\ldots,N-1$) for the relation given by Formula~\eqref{tower}. 

For the small common multiple $\mu_2\circ_1\mu_N\circ_1\mu_N$ the corresponding S-polynomial can be, using the relations $\mu_N\circ_i\mu_N=0$, be reduced to 
\begin{multline*}
(-1)^{N}(\mu_2\circ_2\mu_N)\circ_1\mu_N-\mu_N\circ_1\mu_2\circ_1\mu_N=
(-1)^{N+N^2}(\mu_2\circ_1\mu_N)\circ_{N+1}\mu_N-\\-\mu_N\circ_1\mu_2\circ_1\mu_N\rightarrow
(-1)^{N}(\mu_N\circ_N\mu_2)\circ_{N+1}\mu_N-(-1)^{N}(\mu_N\circ_1\mu_2)\circ_2\mu_N,                                                             
\end{multline*}
where the first equality comes from the graded associativity of the operad composition, and the second one uses reductions modulo the defining relations~\eqref{mu2N} and~\eqref{muNN}. For $1<i\le N$, the small common multiple $\mu_2\circ_1\mu_N\circ_i\mu_N$ gives rise to the S-polynomial
\begin{equation}\label{Spolyformovingcorollas}
\mu_N\circ_{i-1}(\mu_2\circ_2\mu_N)-\mu_N\circ_i(\mu_2\circ_1\mu_N), 
\end{equation}
which can be reduced modulo the defining relations~\eqref{mu2N} and~\eqref{muNN} to
\begin{equation}\label{move-temp}
\mu_N\circ_{i-1}(\mu_2\circ_2\mu_N)-(-1)^N\mu_N\circ_i(\mu_2\circ_2\mu_N). 
\end{equation}
Altogether these elements can be reduced to zero using the relations 
 $$
\mu_N\circ_i(\mu_2\circ_2\mu_N)=(-1)^{N(i-1)}\mu_N\circ_N(\mu_2\circ_2\mu_N), 
 $$
which are precisely the relations $R_{i,1}$ of what we want to prove to be the reduced Gr\"obner basis. The S-polynomials corresponding to the small common multiples of the monomial relations $\mu_N\circ_i\mu_N=0$ and $\mu_N\circ_j\mu_N=0$ are trivially zero, so there is nothing to check.

All the S-polynomials corresponding to the small common multiples of the leading terms the defining relations have now been treated. Let us now study the small common multiples of the leading terms of the defining relations with the leading terms of relations~\eqref{tower}.

All the small common multiples of the leading terms of the relations $R_{i,1}$ with the monomial relations $\mu_N\circ_j\mu_N$ can be easily reduced to zero without using any new relations, as well as the small common multiple $(\mu_N\circ_i(\mu_2\circ_2\mu_N))\circ_i\mu_N$ of the leading term of the relation $R_{i,1}$ with the leading term of $\mu_2\star\mu_N+\mu_N\star\mu_2=0$.

It is also easy to see that there are no new relations needed to reduce all the S-polynomials arising from the small common multiples of the leading terms of relations $R_{i,1}$ and $R_{j,1}$. Indeed, there are two combinatorially different kinds of small common multiples of that sort. The small common multiple 
 $$
(\mu_N\circ_i(\mu_2\circ_2\mu_N))\circ_{j+N}(\mu_2\circ_2\mu_N)=(-1)^N(\mu_N\circ_j(\mu_2\circ_2\mu_N))\circ_{i}(\mu_2\circ_2\mu_N) 
 $$
for $1\le i<j\le N-1$ leads to the S-polynomial 
 $$
(-1)^{N(i-1)}(\mu_N\circ_N(\mu_2\circ_2\mu_N))\circ_{j+N}(\mu_2\circ_2\mu_N)-(-1)^{N+N(j-1)}(\mu_N\circ_N(\mu_2\circ_2\mu_N))\circ_i(\mu_2\circ_2\mu_N)
 $$
which can be reduced to zero modulo the relations $R_{k,1}$ for various~$k$ (taking into the account the graded associativity of the operad composition). The small common multiple
 $$
(\mu_N\circ_i(\mu_2\circ_2\mu_N))\circ_{i+j}(\mu_2\circ_2\mu_N) 
 $$
for $1\le i\le j\le N-1$ leads to the S-polynomial 
 $$
(-1)^{N(i-1)}(\mu_N\circ_N(\mu_2\circ_2\mu_N))\circ_{i+j}(\mu_2\circ_2\mu_N)-(-1)^{N(j-1)}\mu_N\circ_i(\mu_2\circ_2(\mu_N\circ_N(\mu_2\circ_2\mu_N))) 
 $$
where the second term is immediately reduced to 
\begin{multline*}
(-1)^{N(j-1)+N(i-1)+N(i-1)}\mu_N\circ_N(\mu_2\circ_2(\mu_N\circ_N(\mu_2\circ_2\mu_N)))=\\=
(-1)^{N(j-1)}\mu_N\circ_N(\mu_2\circ_2(\mu_N\circ_N(\mu_2\circ_2\mu_N)))
\end{multline*}
using the relation $R_{i,1}$ twice, while the first term is reduced to the same result through a lengthier sequence of reductions (depending on $i+j$ being less than, equal to or greater to $N$). 

The small common multiple $(\mu_N\circ_i(\mu_2\circ_2\mu_N))\circ_i\mu_2$ of the leading term of the relation $R_{i,1}$ with the leading term of $\mu_2\star\mu_2=0$ creates an S-polynomial that can be reduced to $R_{i,2}$, and hence can be reduced to zero using all the elements we have. 

The last computation at this stage is that for the S-polynomial coming from a yet another small common multiple of the leading terms of the relations $R_{i,1}$ and $\mu_2\star\mu_N+\mu_N\star\mu_2=0$, that is $\mu_2\circ_1\mu_N\circ_i(\mu_2\circ_2\mu_N)$. However, it is true that this S-polynomial can be reduced completely using the defining relations of the operad together with the relations $R_{i_1}$ and $R_{i,2}$. 

The way the elements $R_{i,3}$ etc. arise is exactly similar, and our system of elements can be shown to be sufficient to reduce to zero all arising S-polynomials.
\end{proof}

\begin{corollary}\label{basis-B}
A basis for the operad $\NA_{2,N}$ can be defined inductively as follows. The identity map $\id\in\NA_{2,N}(1)$ is a basis element, for every basis element $b$, $\mu_2(\id,b)$ is a basis element, and also for each choice of nonnegative integers $i_1,\ldots,i_{N-1}$, $\mu_N(\mu_2^{(i_1)},\mu_2^{(i_2)},\ldots,\mu_2^{(i_{N-1})},\mu_2(\id,b))$ is a basis element. 
\end{corollary}

\begin{proof}
Indeed, these are precisely the normal tree monomials with respect to the leading terms of the elements of the Gr\"obner basis we computed. 
\end{proof}

\section{Generators and relations for the Yoneda $\A_\infty$-algebras of $N$-Koszul algebras}\label{sec:present}

In this Section, we use the aforementioned operadic notions to provide a natural framework to describe the Koszul dual algebras of $N$-Koszul algebras functorially via generators and relations.

\begin{theorem}\label{thm:presentationKDinfty}
Let $A=T(V)/(R)$ be a finitely generated $N$-homogeneous algebra.
Then
\begin{equation}\label{genrel}
A^!\cong\NA_{2,N}(V^*)/(\mu_2(V^*,V^*),\mu_N(R^\bot)) 
\end{equation}
as  $\NA_{2,N}$-algebras, where $\mu_N(R^\bot)$ is viewed as a subspace of $\mu_N(V^*,V^*,\ldots,V^*)$.
\end{theorem}

\begin{proof} Let us first prove a particular case of this result, and then use it in the general case.
\begin{lemma}
This theorem holds for $R=V^{\otimes N}$. 
\end{lemma}
\begin{proof}
Let us remark that, in this case, the algebra on the right-hand side of~\eqref{genrel} is the algebra $D:=\mathop{\NA_{2,N}}(V^*)/(\mu_2(V^*,V^*))$ of which all the algebras of the right-hand sides of~\eqref{genrel} for various $R$ are quotients. Applying the methods we explained in the end of Section~\ref{sec:Grobner}, we shall study the corresponding extension $\NA_{2,N}\ltimes D$. If we fix some basis $e_1,\ldots,e_k$ of $V^*$, the relations needed to present the extension are $\mu_2(e_i,e_j)=0$. The small common multiple of the associativity relation $\mu_2\star\mu_2=0$ and the relation $\mu_2(e_i,e_j)=0$ produces the relation $\mu_2(e_i,\mu_2(e_j,\id))=0$. This extra relation has no small common multiples with the other relations.

Applying Theorem~\ref{basis-extension}, in the view of Theorem~\ref{thm:GBB2N} and Corollary~\ref{basis-B}, we immediately conclude that a basis for the algebra $D$ can be defined inductively as follows. It has even and odd elements, all generators $e_1$, \ldots, $e_k$ are odd basis elements, for every even basis element $b$, and each $1\le j\le k$, the element $\mu_2(e_j,b)$ is an odd basis element, and also for each odd basis element $b$, and for each $1\le i_1,\ldots,i_{N-1}\le k$, the element $\mu_N(e_{i_1}, e_{i_2},\ldots, e_{i_{N-1}},b)$ is an even basis element. In other words, our algebra has nonzero elements only of weight divisible by~$N$ or congruent to~$1$ modulo~$N$, and for each such weight there exists exactly one type of basis element of that weight. Combinatorially, the corresponding trees are alternating ``towers'' of operations with all the compositions using the last slot of an operation only. This gives a vector space identification of 
$D\cong V^*\oplus \mu_N(V^*, \ldots, V^*)\oplus (\mu_2\circ_2 \mu_N)(V^*, \ldots, V^*)\oplus \cdots$ with the Yoneda $\A_\infty$-algebra 
$A^!= V^*\oplus {V^*}^{\otimes N}\oplus {V^*}^{\otimes (N+1)}\oplus \cdots$ of the algebra $A=T(V)/(V^{\otimes N})$. Comparing the operations of $D$ with those given by Corollary~\ref{A2N-constr}, we see that the corresponding $\NA_{2,N}$-algebras are isomorphic.
\end{proof}

Let us prove the theorem for a general set of relations~$R$. In the proof above for the case of the algebra $D$, corresponding to $R=V^{\otimes N}$, we obtained a basis where we alternate the operations $\mu_2$ and $\mu_N$, computing all compositions at the last slot, and then substitute into the resulting operation an arbitrary word in $e_1$, \ldots, $e_k$. Now, the defining relation $\mu_2\star\mu_N+\mu_N\star\mu_2=0$, together with the vanishing of all the elements~\eqref{tower}, \eqref{Spolyformovingcorollas}, and \eqref{move-temp} mean that in our alternating towers, the only operation we plug in at each level can be freely moved between the slots. If we now impose the additional relations $\mu_N(R^\bot)=0$ and use the identification
$D\cong V^*\oplus {V^*}^{\otimes N}\oplus {V^*}^{\otimes (N+1)}\oplus \cdots$ discussed above, it becomes clear that the underlying vector space of the quotient that we are studying is the direct sum of the appropriate homogeneous components of the algebra~$A^\vee=T(V^*)/(R^\bot)$. The isomorphism of $\NA_{2,N}$-algebras follows, yet again, from Corollary~\ref{A2N-constr}.
\end{proof}

\begin{corollary}
An $N$-homogeneous algebra $A=T(V)/(R)$ is $N$-Koszul if and only if the Yoneda $\A_\infty$-structure on its $\Ext$-algebra factors through the operad $\NA_{2,N}$, and 
 $$
\Ext^\bullet_A(\KK,\KK)\cong\NA_{2,N}(V^*)/(\mu_2(V^*,V^*),\mu_N(R^\bot)) 
 $$
as  $\NA_{2,N}$-algebras.
\end{corollary}

\begin{proof}
This is a direct corollary of Theorem~\ref{computeproducts} and Theorem~\ref{thm:presentationKDinfty}.
\end{proof}

\section{Higher Koszul duality theory}\label{sec:higherKoszul}

This section provides a conceptual framework for the main constructions and statements of the $N$-Koszul duality theory. It proves the equivalence between finding a resolution of the trivial $A$-module and finding a quasi-free algebra resolution of $A$. 

\subsection{Koszul morphism}
Let $(C, \delta_2, \delta_N)$ be an $\A_{2,N}$-coalgebra and let $(A, \mu)$ be an algebra. The following constructions are particular cases of Chapter~$3$ of the Ph.D. thesis of A. Prout\'e \cite{Proute86}.

\begin{definition}
The \emph{convolution $\A_{2,N}$-algebra} is defined by the  space $\Hom(C, A)$ of linear maps from $C$ to $A$ endowed with the following operations 
$$\begin{array}{ccl}
\star_2(f,g)&:=&  C \xrightarrow{\delta_2} C\otimes C \xrightarrow{f \otimes g} A\otimes A \xrightarrow{\mu} A \ , \\
 \star_N(f_1, \ldots, f_N)&:=&  C \xrightarrow{\delta_N} C^{\otimes N} \xrightarrow{f_1 \otimes \cdots \otimes f_N} A^{\otimes N} \xrightarrow{\mu^{(N-1)}} A \ , 
\end{array}$$
where ${\mu}^{(N-1)}$ stands for any $N-1$ iterations of the associative product $\mu\colon A\otimes A\to A$.
\end{definition}

\begin{definition}
Inside the convolution $\A_{2,N}$-algebra, we consider the \emph{Maurer--Cartan equation}: 
$$\star_2(\alpha, \alpha)+ \star_N(\alpha, \ldots  , \alpha)=0 \ .$$
Solutions of degree $-1$ to the Maurer--Cartan equation are called \emph{twisting morphisms}. The associated set is denoted by $\Tw(C,A)$.
\end{definition}

Let $\alpha : C \to A$ be a degree $-1$ linear map. We consider the degree $-1$ derivation $d_\alpha:=d_2+d_N$ on the free right $A$-module $C\otimes A$, where  
$$\begin{array}{ccl}
d_2&:=&  C\otimes A \xrightarrow{\delta_2\otimes A} C \otimes C\otimes A \xrightarrow{C \otimes \alpha  \otimes A} C\otimes A  \otimes A \xrightarrow{C\otimes\mu} C\otimes A \ , \\
d_N&:=&   \ C\otimes A \xrightarrow{\delta_N\otimes A}   C^{\otimes N}\otimes A \xrightarrow{ C\otimes \alpha^{\otimes (N-1)} \otimes A} C\otimes A^{\otimes N} \xrightarrow{\mu^{N-1}\otimes C} C\otimes A \ . 
\end{array}$$

\begin{lemma}
If $\alpha \in\Tw(C,A)$, then ${d_\alpha}^2=0$.
\end{lemma}

\begin{proof}
By straightforward computation.
\end{proof}

\begin{definition}
The chain complex 
$$C\otimes_\alpha A :=(C\otimes A, d_\alpha) $$
associated to a twisting morphism $\alpha$ is called a \emph{twisted tensor product}. 
When this chain complex is acyclic, the twisting morphism $\alpha$ is called a \emph{Koszul morphism}. The set of all Koszul morphisms is denoted by $\Kos(C,A)$. 
\end{definition}

\subsection{Cobar constructions}

The data of an $\A_{2,N}$-coalgebra is equivalent to a square-zero derivation $d$ on a free algebra $T(s^{-1}C)$, on the homological desuspension of $C$, such that the components of the restriction $d_{|s^{-1}C}\colon s^{-1}C \to \bigoplus_{n\ge 1} (s^{-1}C)^{\otimes n}$ vanish for $n\ne 2, N$. 

\begin{definition}
We denote this quasi-free dg algebra by 
$\Omega_\infty C :=(T(s^{-1}C), d)$ and call it the \emph{cobar construction of $C$}.
\end{definition}

Because of the universal property of free algebras, defining a morphism of graded algebras from $T(s^{-1}{C})$ to $A$ is equivalent to defining a degree~$-1$ map of graded vector spaces from $C$ to $A$. Inside the space $\Hom_{\mathsf{gr}\ \mathsf{alg}}\left(T(s^{-1}{C}),\, A\right)$ of all morphisms of graded algebras, there is a subspace $\Hom_{\mathsf{dg}\ \mathsf{alg}}\left(\Omega_\infty C,\, A\right)$ consisting of morphisms of dg algebras. It includes an even smaller subspace $\mathop{\mathrm{QI}}_{\mathsf{dg}\ \mathsf{alg}}\left(\Omega_\infty C,\, A\right)$ consisting of quasi-isomorphisms of dg algebras. 

\begin{theorem}\label{thm:FunThmKosMor}Suppose that both $C$ and $A$ are connected with respect to the weight grading. 
Under the identification  $\Hom_{\mathsf{gr}\ \mathsf{alg}}\left(T(s^{-1}{C}),\, A\right)  \cong \Hom({C},\, A)_{-1}$, we have
$$
\begin{array}{c c c  }
 \Hom_{\mathsf{gr}\ \mathsf{alg}}\left(T(s^{-1}{C}),\, A\right)  &  \cong  &  \Hom({C},\, A)_{-1}  \\
&& \\
 \bigcup &&  \bigcup   \\
&&\\
 \Hom_{\mathsf{dg}\ \mathsf{alg}}\left(\Omega_\infty C,\, A\right)  &   \cong   &   \mathrm{Tw}(C,\, A)   \\
&&\\
 \bigcup  &&  \bigcup   \\
&&\\
 {\mathrm{QI}}_{\mathsf{dg}\ \mathsf{alg}}\left(\Omega_\infty C,\, A\right)  &   \cong  &   \mathrm{Kos}(C,\, A) \ . \\
&&
\end{array}
$$
\end{theorem}

\begin{proof}
We denote the unique morphism of graded algebras associated to a degree $-1$ linear map $\alpha : C \to A$ by 
$g_\alpha : T(s^{-1}{C}) \to A$. It is easy to check that it commutes with the differential of the cobar construction if and only if $\alpha$ is a twisting morphism. 

Let us prove the last equivalence. First, any $\A_{(2, N)}$-coalgebra $(C, \delta_2, \delta_N)$ admits a canonical twisting morphism
$$\iota : C \to s^{-1}C \hookrightarrow \Omega_\infty C=(T(s^{-1}{C}), d)\  .$$ 
It satisfies the following universal property: any twisting morphism $\alpha \in \Tw(C, A)$ factors through it
$$\xymatrix{
 & \Omega_\infty C \ar[dr]^{g_\alpha} &  \\
C \ar[rr]^{\alpha}\ar[ur]^{\iota}  &   &  A  \ .
}$$
Hence, the tensor product of the map $g_\alpha$ with the identity map on $C$ is a map of chain complexes
$$C \otimes_\iota \Omega_\infty C \xrightarrow{g_\alpha \otimes C} C\otimes_\alpha A \  .$$
Under the connectedness assumptions, we can apply the same arguments as the comparison lemma for twisted tensor products \cite[Lemma~$2.5.1$]{LodayVallette10}.  The twisted tensor product $C \otimes_\iota \Omega_\infty C$ is acyclic by \cite[Theorem~$(3.19)$]{Proute86}. Therefore, the map $g_\alpha$ is a quasi-isomorphism if and only if the twisted tensor product $C\otimes_\alpha A$ is acyclic. 
\end{proof}

This theorem demonstrates that the existence of a  quasi-free algebra resolution 
$\Omega_\infty C \xrightarrow{\sim} A$ 
of an  algebra $A$ is equivalent to the existence of 
a  free $A$-module resolution $C\otimes_\alpha A \xrightarrow{\sim} \KK$ of the ground field $\KK$. So, given an algebra $A$, for instance an $N$-homogeneous algebra, one has to look for an $\A_{2,N}$-coalgebra $C$ (or more generally an $\A_\infty$-coalgebra) together with a Koszul morphism $\alpha \, : \, C \to A$ 
to solve these two questions.

The bar construction $\BB A$ provides a functorial dg coalgebra which induces functorial resolutions. 

\begin{proposition} [{\cite{HMS74}}]
For any augmented algebra $A$, the bar-cobar construction $\Omega \BB A \xrightarrow{\sim} A$ is a functorial resolution. The twisted tensor product $\BB A \otimes_\pi A \xrightarrow{\sim}  \KK$, associated to the twisting morphism $\pi \, : \, \BB A =T^c(s\bar A)\twoheadrightarrow s \bar A \to A $, is a resolution of the ground field $\KK$.
\end{proposition}

This general idea to reduce to size of these two resolutions is to consider the underlying homology groups $H_\bullet(\BB A)$ instead of the bar construction $\BB A$ together with the transferred $\A_\infty$-coalgebra structure. This is precisely what the higher Koszul duality does, by proposing a candidate for $H_\bullet(\BB A)$.

\subsection{Higher Koszul duality theory}
Let  $A=T(V)/(R)$ be an $N$-homogeneous algebra. We view $A^{\ac}$ as a $\A_{2,N}$-coalgebra dual to the algebra obtained in Theorem~\ref{thm:presentationKDinfty}.

\begin{definition}
We consider the degree $-1$ linear map 
$\kappa : A^{\ac} \to A$ 
defined by 
$$\kappa : A^{\ac} \twoheadrightarrow A^{\ac}_1 \cong  V  \hookrightarrow A \ .$$ 
\end{definition} 

\begin{lemma}
The map $\kappa \in \Hom(A^{\ac}, A)_{-1}$ is a twisting morphism. 
\end{lemma}

\begin{proof}
Since $A^{\ac}_2=0$, we automatically have $\star_2( \kappa,  \kappa)=0$.
The image of $A^{\ac}_N=R$ under $\star_N( \kappa, \ldots,  \kappa)$  is equal to zero, since the following composite vanishes 
$$A^{\ac}_N=R \xrightarrow{\delta_N} V^{\otimes N} \xrightarrow{\mu^{N-1}} A \ . $$
For weight grading reasons, the other components of 
$\star_2( \kappa,  \kappa)+ \star_N( \kappa, \ldots  ,  \kappa)$ vanish, which concludes the proof. 
\end{proof}

The twisting morphism $\kappa$ can be used to explain the particular form of the Koszul complex introduced by Berger in~\cite{Berger01}. 

\begin{proposition}
The twisted tensor product $A^{\ac}\otimes_\kappa A:=(A^{\ac}\otimes A, d_\kappa)$ coincides with the \emph{Koszul complex} described in Section~\ref{subsec:N-Koszul}.
\end{proposition}

\begin{proof}
By straightforward computation.
\end{proof}

\begin{theorem}\label{thm:KDthm}
The morphism of dg algebras 
$ g_\kappa\colon \Omega_\infty A^{\ac} \to A $
is a quasi-isomorphism if and only if the Koszul complex 
$A^{\ac} \otimes_\kappa A$ is acyclic.
\end{theorem}

\begin{proof}
This is a direct corollary of Theorem~\ref{thm:FunThmKosMor}.
\end{proof}

This theorem demonstrates that the Koszul dual $\A_{2,N}$-coalgebra construction of Theorem~\ref{thm:presentationKDinfty} provides a good candidate to get  a ``small'' quasi-free algebra resolution of an $N$-homogeneous algebra $A$ and to get  an equally ``small'' free $A$-module resolution of the ground field $\KK$. In the Koszul duality theory, this latter resolution is used to compute $\Tor$ and $\Ext$ functors, see \cite{Priddy70} and to establish equivalence between derived categories, see, e.~g., \cite{BGS88} and~\cite{Keller03}. The data of a quasi-free algebra resolution is used in the Homotopy Transfer Theorem \cite[Section~$10.3$]{LodayVallette10}. 

\bibliographystyle{amsalpha}

\def\cprime{$'$}
\providecommand{\bysame}{\leavevmode\hbox to3em{\hrulefill}\thinspace}
\providecommand{\href}[2]{#2}

\end{document}